\documentclass[11pt,onecolumn,twoside]{IEEEtran} 
\usepackage{graphicx}
\usepackage{amsmath}
\usepackage{mathrsfs}
\usepackage{mathtools}
\usepackage{amssymb}
\usepackage{float}
\usepackage{url}
\usepackage{verbatim}
\usepackage{amsthm}
\usepackage{enumerate}
\usepackage{layout}
\usepackage{bm}

\newtheorem{theorem}{Theorem}
\newtheorem{lemma}{Lemma}
\newtheorem{proposition}{Proposition}

\newtheorem{definition}{Definition}
\theoremstyle{definition}

\newtheorem{remark}{Remark}

\newcommand{\naturals}{\ensuremath{\mathbb{N}}}
\newcommand{\Reals}{\ensuremath{\mathbb{R}}}
\newcommand{\expectation}{\ensuremath{\mathbb{E}}}

\newcommand{\set}{\ensuremath{\mathcal}}

\newcommand{\dif}{\mathrm{d}}
\newcommand{\TV}{d_{\mathrm{TV}}}

\begin{document}
\thispagestyle{empty}
\setcounter{page}{1}
\setlength{\baselineskip}{1.15\baselineskip}

\title{\huge{Lower Bounds for the Total Variation Distance Given Means and Variances of Distributions}\\[0.2cm]}
\author{Tomohiro Nishiyama\\ Email: htam0ybboh@gmail.com}
\date{}
\maketitle
\thispagestyle{empty}

\begin{abstract}
For arbitrary two probability measures on real d-space with given means and variances (covariance matrices), we provide lower bounds for their total variation distance. In the one-dimensional case, a tight bound is given.
\end{abstract}
 
\section{Introduction}
The $f$-divergence~\cite{sason2016f} is an important class of divergence measures, defined by means of convex functions $f$, includes many important divergences such as the total variation (TV) distance and the Kullback-Leibler (KL) divergence~\cite{kullback1951information}. Given means and variances of two probability measures, closed-form lower bounds for the $f$-divergence are useful in practice because they can be directly evaluated based on only moments of distributions. These bounds are now beginning to be applied to nonequilibrium physics~\cite{hasegawa2021irreversibility, van2020unified, falasco2022beyond}.

For probability measures on real line, the tight lower bound on the $\chi^2$-divergence is known as the 
Hammersley–Chapman–Robbins bound~\cite{chapman1951minimum, hammersley1950estimating}. In our previous works, we provided tight lower bounds for the KL-divergence and the squared Hellinger distance~\cite{nishiyama2020relations,nishiyama2020tight}. We generalized these bounds for the asymmetric $\alpha$-divergence~\cite{cichocki2010families} with $\alpha\in[-1,2]$~\cite{nishiyama2021tight}, which include the above three divergences. We also provided a sufficient condition that an arbitrary symmetric $f$-divergence has a closed-form expression for lower bound~\cite{nishiyama2022relations}.
Although the TV distance is a symmetric $f$-divergence, it does not satisfy this sufficient condition. On the other hand, for the TV distance between single Gaussians or Gaussian mixtures on real $d$-space, the closed-form expressions for lower bounds have recently been derived~\cite{davies2022lower, devroye2018total}. 

In this work, we study closed-form lower bounds for the TV distance between arbitrary probability measures on real $d$-space with given means and variances (covariance matrices). In the one-dimensional case, we provide a tight lower bound.

\section{Main Results}
\subsection{Preliminaries}
We provide definitions which are used in this paper.

\begin{definition}
Let $P$ and $Q$ be probability measures defined on a common measurable space $(\mathcal{A}, \mathscr{F})$. Let $\mu$ be a dominating measure
of $P$ and $Q$ (i.e., $P, Q \ll \mu$), and let $p := \frac{\mathrm{d}P}{\mathrm{d}\mu}$
and $q := \frac{\mathrm{d}Q}{\mathrm{d}\mu}$ be the densities of $P$ and $Q$ with respect
to $\mu$. The {\em total variation (TV) distance} between $P$ and $Q$ is given by
\begin{align} 
\TV(P,Q) := \frac12 \int |p-q|\dif \mu = \sup_{\mathcal{F}\in \mathscr{F}} |P(\mathcal{F})-Q(\mathcal{F})|. \nonumber
\end{align}
\end{definition}

\begin{definition} 
Let $P$ and $Q$ be probability measures on $\Reals$. Let $m_P, m_Q, \sigma_P^2$, and $\sigma_Q^2$ denote the means and the variances of $X\sim P$ and $Y\sim Q$, i.e.,
\begin{align}
&m_P:=\expectation[X], \; m_Q:=\expectation[Y], \nonumber  \\    
&\sigma_P^2:=\expectation[(X-m_P)^2], \; \sigma_Q^2:=\expectation[(Y-m_Q)^2]. \nonumber
\end{align}
A set of pairs of probability measures $(P,Q)$ with given means $(m_P, m_Q)$ and variances $(\sigma_P^2, \sigma_Q^2)$ is defined as $\set{P}[m_P, \sigma_P; m_Q, \sigma_Q]$. 
Similarly, a set of pairs of probability measures $(P,Q)$ on $\Reals^d$ with given means $(\bm{m}_P, \bm{m}_Q)$ and covariance matrices $(\Sigma_P, \Sigma_Q)$ is defined as $\set{P}[\bm{m}_P, \Sigma_P; \bm{m}_Q, \Sigma_Q]$, where $\Sigma_P:=\expectation[(\bm{X}-\bm{m}_P)(\bm{X}-\bm{m}_P)^T]$ and $\Sigma_Q:=\expectation[(\bm{Y}-\bm{m}_Q)(\bm{Y}-\bm{m}_Q)^T]$.
\end{definition}

\subsection{Lower bounds for the TV distance}
Our main result is the following tight lower bound between a pair of arbitrary probability measures on real line with given means and variances.
\begin{theorem} \label{th_1}
Let $(P, Q)\in \set{P}[m_P, \sigma_P; m_Q, \sigma_Q]$. 
\begin{enumerate}[(a)]
\item
If $m_P\neq m_Q$, then
\begin{align}
\label{th_1_1}
\TV(P,Q)\geq \frac{a^2}{(\sigma_P+\sigma_Q)^2+a^2}, 
\end{align}
where $a:=m_P-m_Q$. 
\item 
The lower bound on the right side of \eqref{th_1_1} is attained by a pair of two or three-element probability distributions as follows: 
\begin{align}
&P=(1-p\;, p, \;0), \quad Q=(1-p,\; 0, \;p),  \quad \mbox{if} \hspace*{0.15cm}  \sigma_P>0  \hspace*{0.15cm} \mbox{and} \hspace*{0.15cm} \sigma_Q>0, \nonumber \\
&P=(1-p\;, p), \quad Q=(1,\; 0),  \quad \mbox{if} \hspace*{0.15cm}  \sigma_P>0 \hspace*{0.15cm} \mbox{and} \hspace*{0.15cm} \sigma_Q=0, \nonumber \\
&P=(1\;, 0), \quad Q=(1-p,\; p),  \quad \mbox{otherwise} , \nonumber 
\end{align}
where 
\begin{align}
p:=\frac{a^2}{(\sigma_P+\sigma_Q)^2+a^2}\in(0,1]. \nonumber
\end{align}
\item
If $m_P=m_Q$, then
\begin{align}
\label{th_1_3}
\inf\TV(P,Q)=0.
\end{align}
\end{enumerate}
The infimum in \eqref{th_1_3} are taken over all $(P,Q)\in \set{P}[m_P, \sigma_P; m_Q, \sigma_Q]$.
\end{theorem}
\begin{proof}
See Section~\ref{section: proofs}.
\end{proof}
\begin{remark}
For KL-divergence, $\chi^2$-divergence and squared Hellinger distance, a pair of probability measures defined on a common two-point set attains lower bounds with given means and variances. However, this does not hold for the TV distance when $\sigma_P$ and $\sigma_Q$ are positive.
\end{remark}
We next provide a lower bound for probability measures on $\Reals^d$.
\begin{proposition}
\label{prop_1}
Let $P$ and $Q$ be probability measures on $\Reals^d$, and let $(P,Q)\in\set{P}[\bm{m}_P, \Sigma_P; \bm{m}_Q, \Sigma_Q]$. Then, 
\begin{align}
\label{prop_1_1}
\TV(P,Q)\geq \frac{\bm{a}^T \bm{a}}{2(\mathrm{tr}(\Sigma_P)+\mathrm{tr}(\Sigma_Q))+\bm{a}^T \bm{a}},
\end{align}
where $\mathrm{tr}(A)$ denotes the trace of a matrix $A$, and $\bm{a}:=\bm{m}_P-\bm{m}_Q$.
\end{proposition} 
\begin{proof}
Since the TV distance is invariant under transformation $\bm{x}\rightarrow \bm{x}-\frac{\bm{m}_P+\bm{m}_Q}2$, one can assume $(P,Q)\in \set{P}[\frac{\bm{a}}2, \Sigma_P; -\frac{\bm{a}}2, \Sigma_Q]$ without any loss of generality. 
By the Cauchy-Schwarz inequality, we have 
\begin{align}
\label{pr_1_1}
(\int |p-q||x_k| \dif\mu)^2=(\int \sqrt{|p-q|}\sqrt{|p-q|}|x_k| \dif\mu)^2 \leq  \int |p-q|\dif\mu \int (p+q)x_k^2  \dif\mu \nonumber \\
= 2\TV(P,Q)\Bigl(({\Sigma_P})_{kk}^2 + ({\Sigma_Q})_{kk}^2+\frac{a_k^2}2\Bigr).
\end{align}
By combining this inequality with $\int |p-q||x_k| \dif\mu\geq  |\int (p-q)x_k \dif\mu|=|a_k|$, it follows that 
\begin{align}
2\TV(P,Q) \Bigl(({\Sigma_P})_{kk}^2 + ({\Sigma_Q})_{kk}^2+\frac{a_k^2}2\Bigr) \geq a_k^2.
\end{align}
Taking the sum over $1\leq k \leq d$ for this inequality yields \eqref{prop_1_1}.
\end{proof}

\section{Proof of Theorem~\ref{th_1}}
\label{section: proofs}

\subsection{Proofs of Lemmas}
\label{sub_sec:minimum_condition}
Let $\set{P}_n$ be a set of pairs of probability measures defined on a common $n$-point set $\{x_1, x_2, \cdots, x_n\}$, where $\{x_i\}_{1\leq i\leq n}$ are arbitrary real numbers. Let $\set{P}_3^*\subset\set{P}_3$ be a set of pairs of probability measures $P$ and $Q$ such that $P=(1-p, p,0)$ and $Q=(1-p, 0,p)$ for $p \in [0,1]$. 
Before proving Theorem~\ref{th_1}, we prove the following lemmas.
\begin{lemma} \label{lem_pstar}
Let $\sigma_P$ and $\sigma_Q$ be positive, and let $m_P\neq m_Q$. Then, 
\begin{align}
\label{lem_pstar_1}
\min_{(P,Q)\in\set{P}_3^*\cap\set{P}[m_P, \sigma_P; m_Q, \sigma_Q]}\TV(P,Q)=\frac{a^2}{(\sigma_P+\sigma_Q)^2+a^2}.
\end{align}
\end{lemma}
\begin{proof}

The moment constraints reduce to 
\begin{align}
\label{lem_1_1}
\begin{cases}
&(1-p) x_1 + px_2=m_P, \\
&(1-p) x_1^2 + px_2^2=m_P^2+\sigma_P^2,  \\ 
&(1-p) x_1 + px_3=m_Q,  \\ 
&(1-p) x_1^2 + px_3^2=m_Q^2+\sigma_Q^2,
\end{cases}
\end{align}
where $p:=P(x_2)=Q(x_3)$.
Subtracting the square of the first equation in \eqref{lem_1_1} from its second equation, we have
\begin{align}
\label{lem_1_2}
x_1-x_2=\pm\sigma_P\sqrt{\frac{1}{p(1-p)}}.
\end{align}
We first consider the first option in \eqref{lem_1_2}. Solving simultaneously this relation and the first equation in \eqref{lem_1_1}, we obtain 
\begin{align}
\label{lem_1_3}
x_1=m_P+\sigma_P\sqrt{\frac{p}{1-p}}, \\
\label{lem_1_4}
x_2=m_P-\sigma_P\sqrt{\frac{1-p}{p}}.
\end{align}
Similarly, from the third and the forth equation in \eqref{lem_1_1}, we obtain 
\begin{align}
\label{lem_1_5}
x_1=m_Q\mp\sigma_Q\sqrt{\frac{p}{1-p}}, \\
\label{lem_1_6}
x_3=m_Q\pm\sigma_Q\sqrt{\frac{1-p}{p}}.
\end{align}
By combining \eqref{lem_1_3} with the first option in \eqref{lem_1_5},  we have
\begin{align}
-\frac{a}{\sigma_P+\sigma_Q}=\sqrt{\frac{p}{1-p}}, \quad a<0.
\end{align}
Solving this equation for $p$, we have 
\begin{align}
\label{lem_1_7}
p=\frac{a^2}{(\sigma_P+\sigma_Q)^2+a^2}\in(0,1).
\end{align}
It can be verified that $x_1$, $x_2$, and $x_3$ are different from each other from \eqref{lem_1_3}-\eqref{lem_1_6}.
Since $\TV(P,Q)=p$, we have 
\begin{align}
\label{lem_1_8}
\TV(P,Q)=\frac{a^2}{(\sigma_P+\sigma_Q)^2+a^2}.
\end{align}
Similarly, by combining \eqref{lem_1_3} with the second option in \eqref{lem_1_5}, we have
\begin{align}
\label{lem_1_9}
p=\frac{a^2}{(\sigma_P-\sigma_Q)^2+a^2}\in(0,1], \quad \frac{a}{\sigma_P-\sigma_Q}<0.
\end{align}
Thus, we have 
\begin{align}
\label{lem_1_10}
\TV(P,Q)=\frac{a^2}{(\sigma_P-\sigma_Q)^2+a^2}.
\end{align}

We next consider the second option in \eqref{lem_1_2}. By replacing $\sigma_P\rightarrow -\sigma_P$ in \eqref{lem_1_8} and \eqref{lem_1_10}, we obtain the same results for $a>0$ and $ \frac{a}{\sigma_P-\sigma_Q}>0$, respectively. One can verify that these solutions satisfy \eqref{lem_1_1}. 
\end{proof}

\begin{lemma} \label{lem_p2}
Let $m_P\neq m_Q$. Then, a set $\set{P}_{2}\cap \set{P}[m_P, \sigma_P; m_Q, \sigma_Q]$ contains one component $(P,Q)$, and 
\begin{align}
\label{lem_p2_1}
\TV(P,Q)=\frac{a^2}{v},
\end{align}
where 
\begin{align}
\label{v}
v:= \sqrt{(\sigma_Q^2-\sigma_P^2)^2 + 2a^2(\sigma_P^2+\sigma_Q^2)+a^4}.
\end{align}
\end{lemma}
\begin{proof}

The moment constraints reduce to 
\begin{align}
\begin{cases}
\label{lem_2_1}
&(1-p) x_1 + px_2=m_P,  \\ 
&(1-p) x_1^2 + px_2^2=m_P^2+\sigma_P^2,   \\ 
&(1-q) x_1 + qx_2=m_Q,  \\ 
&(1-q) x_1^2 + qx_2^2=m_Q^2+\sigma_Q^2,
\end{cases}
\end{align}
where $p:=P(x_2)$ and $q:=Q(x_2)$.
In the similar way to the proof of Lemma~\ref{lem_pstar}, we have 
\begin{align}
\label{lem_2_2}
x_1=m_P+\sigma_P\sqrt{\frac{p}{1-p}}, \\
\label{lem_2_3}
x_2=m_P-\sigma_P\sqrt{\frac{1-p}{p}}.
\end{align}
It should be noted that an another solution $x_1=m_P-\sigma_P\sqrt{\frac{p}{1-p}}$, $x_2=m_P+\sigma_P\sqrt{\frac{1-p}{p}}$ corresponds to switching $(p,q, x_1)$ and $(1-p, 1-q, x_2)$, and it gives the same probability measure.
From these equations, we obtain 
\begin{align}
\label{lem_2_4}
x_1-x_2=\sigma_P\sqrt{\frac1{p(1-p)}}.
\end{align}
Subtracting the third equation in \eqref{lem_2_1} from its first equation, and subtracting the forth equation in \eqref{lem_2_1} from its second equation, we have
\begin{align}
\label{lem_2_5}
(p-q)(x_2-x_1)&=a, \\
\label{lem_2_6}
(p-q)(x_2-x_1)(x_1+x_2)&=a(x_1+x_2)=a(m_P+m_Q)+\sigma_P^2-\sigma_Q^2.
\end{align}
By substituting \eqref{lem_2_2} and \eqref{lem_2_3} into \eqref{lem_2_6}, we have
\begin{align}
\frac{2p-1}{\sqrt{p(1-p)}}=\frac{\sigma_P^2-\sigma_Q^2-a^2}{a\sigma_P}.
\end{align}
Solving this equation for $p$ gives
\begin{align}
\label{p}
p=\frac12 + \frac{|a|}{a}\frac{(\sigma_P^2-\sigma_Q^2-a^2)}{2v} \in [0,1].
\end{align} 
Substituting \eqref{lem_2_4} and \eqref{p} into \eqref{lem_2_5}, we obtain 
\begin{align}
\label{q}
q =p+ \frac{a|a|}{v}\in [0,1].
\end{align}
The TV distance between $P$ and $Q$ is given by
\begin{align}
\TV(P,Q)=|p-q|=\frac{a^2}{v}.
\end{align}
Hence, we obtain \eqref{lem_p2_1}.
\end{proof}

\begin{lemma} \label{minimum_condition}
\label{lem_minimum}
For $R>0$, let $\set{P}_{n, R}\subset \set{P}_{n}$ be a set of pairs of probability measures such that $|x_i|\leq R$ for all $i=1,2,\cdots, n$.
Let $(P,Q)\in \set{P}_{n, R}\cap \set{P}[m_P, \sigma_P; m_Q, \sigma_Q]$. If $m_P\neq m_Q$, the global minimum point $(P^*, Q^*, x^*)=\mathrm{argmin} \;\TV(P,Q)$ satisfies any one of the following conditions.
\begin{itemize}
\item
$\TV(P^*,Q^*)=\frac{a^2}{(\sigma_P+\sigma_Q)^2+a^2}$ and $\max_{1\leq i\leq n} |x^*_i| < R$.\\
When $\sigma_P>0$ and $\sigma_Q>0$, $(P^*,Q^*)\in\set{P}_3^*$ and otherwise, $(P^*,Q^*)\in\set{P}_2$.
\item
$\max_{1\leq i\leq n} |x^*_i| = R$.
\end{itemize}
\end{lemma}
\begin{proof}

Let $\sigma_P$ or $\sigma_Q$ is positive since the case when $\sigma_P=\sigma_Q=0$ is trivial.
Consider the following minimization problem.
\begin{align}
&\mbox{minimize} \quad 2\TV(P,Q)=\sum_{i} |p_i-q_i| \nonumber \\[0.1cm]
\quad \mbox{subject to}\quad &g_k(p,x):= \sum_i p_i x_i^{k-1} -A_k=0, \nonumber \\[0.1cm] 
& g_{k+3}(q, x):= \sum_i q_i x_i^{k-1} -B_k=0, \quad \mbox{for} \hspace*{0.15cm}  k=1,2,3,\nonumber \\[0.1cm]
&0\leq p_i \leq 1, \quad 0\leq q_i \leq 1, \quad |x_i|\leq R, \quad \mbox{for} \hspace*{0.15cm}  1\leq i\leq n, \nonumber 
\end{align}
where $A:= (1, m_P, m_P^2 + \sigma_P^2)^\mathrm{T}$ and $B:= (1, m_Q,  m_Q^2+ \sigma_Q^2)^\mathrm{T}$. Since the feasible region is compact and the objective function is continuous, there exists a global minimum. In the following, we prove the case when $\max_{1\leq i\leq n} |x^*_i| < R$. 

We first consider the case when $\sigma_P > 0$ and $\sigma_Q > 0$, then we have $p_i ,q_i< 1$ for all $1\leq i\leq n$. For the global minimum point $(P^*, Q^*, x^*)$, we define sets of subscripts as 
\begin{align}
\begin{cases}
&N_{p,q}:= \{i \hspace*{0.15cm} |\hspace*{0.15cm}  p^*_i\neq q^*_i, \; p^*_i>0,  \; q^*_i > 0\}, \nonumber \\
&N_{p=q}:= \{i \hspace*{0.15cm} |\hspace*{0.15cm} p^*_i =q^*_i > 0\}, \nonumber \\
&N_{q=0}:= \{i \hspace*{0.15cm} |\hspace*{0.15cm} p^*_i>0,  \; q^*_i = 0\}, \nonumber \\
&N_{p=0}:= \{i \hspace*{0.15cm} |\hspace*{0.15cm} p^*_i=0,  \; q^*_i > 0\}, \nonumber 
\end{cases}
\end{align}
where $N_{p,q}\bigsqcup N_{p=q}\bigsqcup N_{q=0} \bigsqcup N_{p=0}=\{1,2, \cdots, n\}$. Since the case when $|N_{p,q}|=|N_{p=q}|=0$ must not be minimum, we suppose that $|N_{p,q}|$ or $|N_{p=q}|$ is positive in the following. For sufficiently small $\epsilon>0$, the point $(P^*, Q^*, x^*)$ is a global minimum in a region such that $p_i \in[p^*_i-\epsilon, p^*_i + \epsilon]$ for all $i\in N_{p,q}\bigsqcup N_{p=q} \bigsqcup N_{q=0}$, and $q_j \in[q^*_j-\epsilon, q^*_j + \epsilon]$ for all $j\in N_{p,q}\bigsqcup  N_{p=0}$.
For constraints $\{g_k\}_{1\leq k\leq6}$, we provide the following Lemma. The proof is shown in Appendix~\ref{pr_appendix}.
\begin{lemma}
\label{lem_constraints}
Let $(P,Q)\in \set{P}_{n}\cap \set{P}[m_P, \sigma_P; m_Q, \sigma_Q]$ for $m_P\neq m_Q$, $\sigma_P>0$, and $\sigma_Q>0$. Then, $\{\nabla_{p,q,x} g_k\}_{1\leq k\leq 6}$ are linearly independent.
\end{lemma}
From this lemma, the point $(P^*, Q^*, x^*)$ must be a stationary point of the following Lagrangian.
\begin{align}
\label{pr_1_5}
L(p, q, x , \lambda, \nu):=\sum_{i\in N_{p,q}} s_i (p_i-q_i)+\sum_{i\in N_{q=0}} p_i  + \sum_{i\in N_{p=0}}q_i-\sum_{i\in N_{p,q}\bigsqcup N_{q=0}} p_i\phi_{\lambda}(x_i) \\
-\sum_{i\in N_{p,q}\bigsqcup N_{p=0}} q_i\psi_{\nu}(x_i) - \sum_{i\in N_{p=q}}p_i(\phi_\lambda(x_i)+\psi_\nu(x_i))+\sum_{k=1}^3 \lambda_k A_k+\sum_{k=1}^3 \nu_{k}B_k,
\end{align}
where $s_i=+1$ or $-1$, $\phi_{\lambda}(x):= \sum_{k=1}^3 \lambda_k x^{k-1}$ and  $\psi_{\nu}(x):= \sum_{k=1}^3 \nu_{k} x^{k-1}$.
At a stationary point, all partial derivatives of the Lagrangian with respect to $p,q,x, \lambda, \nu$ should be zero. Letting $\phi(x):=\phi_{\lambda^*}(x)$ and $\psi(x):=\psi_{\nu^*}(x)$, we have
\begin{enumerate}[1)]
\item
$i\in N_{p,q}$ :
\begin{align}
\label{pr_1_6}
\phi(x^*_i)-s_i=\psi(x^*_i)+s_i=0,  \\[0.1cm]
\label{pr_1_7}
p^*_i\phi'(x^*_i)+q^*_i\psi'(x^*_i)=0.
\end{align}
From \eqref{pr_1_6}, we have 
\begin{align}
\label{pr_1_7_2}
\phi(x^*_i)+\psi(x^*_i)=0.  
\end{align}  
\item
$i\in N_{p=q}$ :
\begin{align}
\label{pr_1_8}
\phi(x)+\psi(x)=\alpha_i(x-x^*_i)^2.  
\end{align}
\item
$i\in N_{q=0}$ :
\begin{align}
\label{pr_1_9}
\phi(x)=\beta_i(x-x^*_i)^2+1.
\end{align}
\item
$i\in N_{p=0}$ :
\begin{align}
\label{pr_1_10}
\psi(x)=\gamma_i(x-x^*_i)^2+1,
\end{align}
\end{enumerate}
where $'$ denotes the derivative with respect to $x$, and $\alpha_i, \beta_i, \gamma_i$ are real constants. 
Equations \eqref{pr_1_8}-\eqref{pr_1_10} follow since $\phi(x)$ and $\psi(x)$ are at most quadratic functions with respect to $x$.
To simplify the proof, we provide the following lemma.
\begin{lemma}
\label{lem_lagrange}
Let $|N_{p,q}|$ or $|N_{p=q}|$ be positive. Then, all $\phi(x)-1, \psi(x)-1$, and $\phi(x)+\psi(x)$ are not identically zero.
\end{lemma}
We prove the case of $\phi(x)+\psi(x)$. 
If $\phi(x)+\psi(x)=0$, it follows that $\phi'(x)+\psi'(x)=0$.
When $|N_{p,q}|\geq 1$, from \eqref{pr_1_7} and using $p^*_i\neq q^*_i$, we have $\phi'(x^*_i)=\psi'(x^*_i)=0$ for all $i\in N_{p,q}$. 
From \eqref{pr_1_6}, we have $\phi(x)=\eta_i(x-x^*_i)^2+s_i$ and $\psi(x)=\tilde{\eta}_i(x-x^*_i)^2-s_i$, where $\eta_i$ and $\tilde{\eta}_i$ are constants.
By combining these equations, \eqref{pr_1_9} and \eqref{pr_1_10} with $\sum_{i=1}^n p^*_i=\sum_{i=1}^nq^*_i=1$, there exist $j$ and $k$ such that $\phi(x)=\eta_j(x-x^*_j)^2+1$ and $\psi(x)=\tilde{\eta}_k(x-x^*_k)^2+1$. Since $+s_i$ or $-s_i$ is equal to $-1$, we have $\phi(x)=\eta_i(x-x^*_i)^2-1=\eta_j(x-x^*_j)^2+1$ or $\psi(x)=\tilde{\eta}_i(x-x^*_i)^2-1 = \tilde{\eta}_k(x-x^*_k)^2+1$. Notice that this relation does not hold. Hence, $\phi(x)+\psi(x)$ must not be identically zero.
When $|N_{p,q}|=0$ and $|N_{p=q}|\geq 1$, the result follows since \eqref{pr_1_9} and \eqref{pr_1_10} do not satisfy $\phi(x)+\psi(x)=0$
 (it is noted that $|N_{q=0}|\geq 1$ and $|N_{p=0}|\geq 1$). The cases of $\phi(x)-1, \psi(x)-1$ are shown similarly, and these results complete the proof of Lemma~\ref{lem_lagrange}. 

From Lemma~\ref{lem_lagrange} and \eqref{pr_1_7_2}-\eqref{pr_1_10}, it follows that $|N_{p,q}|\leq 2$, $|N_{p=q}|\leq 1$, $|N_{q=0}|\leq 1$, $|N_{p=0}|\leq 1$, and either $|N_{p,q}|$ or $|N_{p=q}|$ is zero (recall that $\phi(x)$ and $\psi(x)$ are at most quadratic functions).
Furthermore, $|N_{q=0}|$ or $|N_{p=0}|$ must be zero when $|N_{p,q}|\geq 1$, since $\phi(x)-1$ or $\psi(x)-1$ is identically zero from \eqref{pr_1_6} if $|N_{q=0}|=|N_{p=0}|=1$. Therefore, we exclude the case when $|N_{p,q}|=1$ since $|N_{q=0}|$ and $|N_{p=0}|$ must be $1$ from $\sigma_P>0$ and $\sigma_Q>0$.
By summarizing these results, the minimum point satisfies one of the following conditions.

\begin{enumerate}[(A)]
\item 
$|N_{p,q}|=0, \;|N_{p=q}|=|N_{q=0}|=|N_{p=0}|=1$: \\
This case is equivalent to $(P^*,Q^*)\in \set{P}_3^*$. From Lemma~\ref{lem_pstar}, we have 
\begin{align}
\label{pr_1_11}
\TV(P^*,Q^*)=\frac{a^2}{(\sigma_P+\sigma_Q)^2+a^2}.
\end{align}
\item 
$|N_{p,q}|=2, \; |N_{p=q}|=|N_{q=0}|=|N_{p=0}|=0$: \\
This case is equivalent to $(P^*,Q^*)\in \set{P}_2$. From Lemma~\ref{lem_p2} and $(\sigma_P+\sigma_Q)^2+a^2> v$, we have 
\begin{align}
\label{pr_1_12}
\TV(P^*,Q^*)=\frac{a^2}{v}> \frac{a^2}{(\sigma_P+\sigma_Q)^2+a^2}.
\end{align}
\item 
$|N_{p,q}|=2, \; |N_{q=0}|=1, \;|N_{p=q}|=|N_{p=0}|=0$:\\
From Lemma~\ref{lem_lagrange} for $\phi(x)-1$, it follows that $s_i=-1$ for all $i\in N_{p,q}$. Then, we have
\begin{align}
\label{pr_1_13}
\TV(P^*,Q^*)=\frac{2a^2}{v+\sigma_P^2-\sigma_Q^2+a^2}> \frac{a^2}{(\sigma_P+\sigma_Q)^2+a^2}.
\end{align}
The derivation of the TV distance is shown in Appendix~\ref{deri_appendix}.
\item 
$|N_{p,q}|=2, \;|N_{p=0}|=1, \;|N_{p=q}|=|N_{q=0}|=0$:
\begin{align}
\label{pr_1_14}
\TV(P^*,Q^*)=\frac{2a^2}{v+\sigma_Q^2-\sigma_P^2+a^2}> \frac{a^2}{(\sigma_P+\sigma_Q)^2+a^2}.
\end{align}
This relation follows by switching $P^*$ and $Q^*$ in Case (C). 

\end{enumerate}
It should be noted that we exclude the case when $|N_{p=q}|=1$ and $|N_{p,q}|=|N_{q=0}|=|N_{p=0}|=0$ from the assumption $m_P \neq m_Q$. From \eqref{pr_1_11}-\eqref{pr_1_14} and Case (A), we complete the proof of Lemma~\ref{minimum_condition} for positive variances.

We next consider the case when $\sigma_P>0$ and $\sigma_Q=0$. Since $|N_{p,q}|$ is positive, we choose $N_{p,q}=\{1\}$, and we have $p_i<1$ for all $1\leq i\leq n$ from $\sigma_P>0$. 
The Lagrangian is given by
\begin{align}
\label{pr_1_15}
L(p, x , \lambda):= 1-p_1+\sum_{i\in N_{q=0}} p_i  -p_1\phi_{\lambda}(m_Q) - \sum_{i\in N_{q=0}} p_i\phi_{\lambda}(x_i) +\sum_{k=1}^3 \lambda_k A_k.
\end{align}
Since $|N_{p,q}|+|N_{q=0}|\geq 2$, $\{\nabla_{p,x} g_k\}_{1\leq k\leq 3}$ are linearly independent.
From $\phi(m_Q)+1=0$ and $\phi(x^*_i)-1=\phi'(x^*_i)=0$ for all $i\in N_{q=0}$, it follows that $|N_{q=0}|\leq 1$. Since this case is equivalent to $(P^*,Q^*)\in \set{P}_2$, we have $\TV(P^*,Q^*)=\frac{a^2}{\sigma_P^2+a^2}$ from Lemma~\ref{lem_p2}.
Switching $P^*$ and $Q^*$ complete the proof. 
\end{proof}
\subsection{Proof of Theorem~\ref{th_1}}

\begin{proof}

We first prove Item (a) and (b) in Theorem~\ref{th_1} for pairs of finite discrete probability measures. \\
Let $t^*:= \inf_{(P,Q)\in\set{P}_n \cap \set{P}[m_P, \sigma_P; m_Q, \sigma_Q]} \TV(P,Q)$, and suppose $t^*<\frac{a^2}{(\sigma_P+\sigma_Q)^2+a^2}$.
Under this assumption, the global minimum is on $\max_i |x^*_i|=R$ for arbitrary $R$ by Lemma~\ref{minimum_condition}. As $R\rightarrow\infty$, there exists a sequence of probability measures $\{P_k\}$ and $\{Q_k\}$ defined on $\{x_{1;k}, x_{2;k} \cdots, x_{n;k}\}$ such that
\begin{align}
\TV(P_\infty, Q_\infty)=t^*,
\end{align}
where $Z_\infty$ denotes $\lim_{k\rightarrow \infty} Z_j$ for variables $Z=\{P,Q, x_i\}$. Without any loss of generality, one can assume that $|x_{i;\infty}| < \infty$ for $1\leq i\leq n'$ and $|x_{i;\infty}| = \infty$ for $n'< i\leq n$. Let $\sum_{i>n'} p_{i;\infty}x^2_{i;\infty}=C^2$ and  $\sum_{i> n'} q_{i;\infty}x^2_{i;\infty}=D^2$, where $p_{i;k}:=P_k(x_{i;k})$ and $q_{i;k}:=Q_k(x_{i;k})$. By the variance constraints, we have $0\leq C^2\leq m_P^2+\sigma_P^2$ and $0\leq D^2\leq m_Q^2+\sigma_Q^2$. Since $p_{i;\infty}=O(x_{i;\infty}^{-2})$ and $q_{i;\infty}=O(x_{i;\infty}^{-2})$ for $i >n'$, it follows that

\begin{align}
\label{eq_moment_1}
\begin{cases}
&\sum_{i>n'} p_{i;\infty}=\sum_{i> n'} p_{i;\infty}x_{i;\infty}=0, \\[0.1cm]
&\sum_{i>n'} q_{i;\infty}=\sum_{i> n'} q_{i;\infty}x_{i;\infty}=0, \\[0.1cm] 
&\sum_{i> n'} |p_{i;\infty}-q_{i;\infty}|=0. 
\end{cases}
\end{align}
Let $P'$ and $Q'$ be probability measures defined on $\{x_{1;\infty}, x_{2;\infty}, \cdots, x_{n';\infty}\}$, and let $P'(x_{i;\infty}):=P(x_{i;\infty})$, $Q'(x_{i;\infty}):=Q(x_{i;\infty})$ for $1\leq i \leq n'$.
From \eqref{eq_moment_1}, it follows that 
\begin{align}
(P',Q')&\in \set{P}_{n',R'} \cap \set{P}[m_P, \sigma_P'; m_Q, \sigma_Q'], \nonumber \\
t^*&=\TV(P', Q'), \nonumber
\end{align}
where we set $R' > \max_{1\leq i\leq n'} |x_{i; \infty}|$, and ${\sigma_P'}^2:=\sigma_P^2-C^2$, ${\sigma_Q'}^2:=\sigma_Q^2-D^2$. If $t^*$ is not a global minimum in $\set{P}_{n',R'} \cap \set{P}[m_P, \sigma_P'; m_Q, \sigma_Q']$, there exists an another sequence in $\set{P}_n \cap \set{P}[m_P, \sigma_P;m_Q, \sigma_Q]$ such that $\sum_{i> n'} p_{i;\infty}x^2_{i;\infty}=C^2$ and $\sum_{i> n'} q_{i;\infty}x^2_{i;\infty}=D^2$, which gives a smaller TV distance than $t^*$. It contradicts that $t^*$ is infimum in $\set{P}_{n} \cap \set{P}[m_P, \sigma_P; m_Q, \sigma_Q]$. Hence, by Lemma~\ref{lem_minimum} for the case when $\max_{1\leq i\leq n'} |x_{i, \infty}|<R'$ and $\sigma_P'\leq \sigma_P$, $\sigma_Q'\leq \sigma_Q$, it follows that
\begin{align}
t^*=\frac{a^2}{(\sigma_P'+\sigma_Q')^2+a^2}\geq \frac{a^2}{(\sigma_P+\sigma_Q)^2+a^2}. \nonumber
\end{align}
Since this contradicts the assumption $t^*<\frac{a^2}{(\sigma_P+\sigma_Q)^2+a^2}$, we have $t^*=\frac{a^2}{(\sigma_P+\sigma_Q)^2+a^2}$. 
Therefore, we obtain \eqref{th_1_1}, and Item (b) follows from Lemma~\ref{lem_minimum}.

We next prove Item (a) for pairs of arbitrary probability measures $(P,Q)\in\set{P}[m_P, \sigma_P; m_Q, \sigma_Q]$.
For sufficiently small $\epsilon > 0$, there exists $M\in \naturals$ such that 
\begin{align}
\label{approximate1}
|\int_{x\in(-\infty, M) \cup [M, \infty)} px^{l} \mathrm{d}\mu|&<\epsilon, \quad |\int_{x\in(-\infty, M) \cup [M, \infty)} qx^{l} \mathrm{d}\mu|<\epsilon, \quad \mbox{for} \hspace*{0.15cm}  l=0,1,2, \nonumber \\
\int_{x\in(-\infty, M) \cup [M, \infty)} |p-q| \dif\mu &< \epsilon.
\end{align}
For a set $S:=\{ x \in[-M, M) |\;p(x)> q(x)\}$, let 
\begin{align}
\begin{cases}
&P_d(\frac{k}{2^M}+\frac1{2^{M+1}} ):=P\Bigl(S \cap \bigl[\frac{k}{2^M}, \frac{k+1}{2^M}\Bigr)\Bigr), \nonumber \\
&P_d(\frac{k}{2^M}+\frac1{2^{M+1}}+\delta ):=P\Bigl(\bigl([-M, M)\setminus S\bigr) \cap \Bigl[\frac{k}{2^M}, \frac{k+1}{2^M}\Bigr)\Bigr), \quad \mbox{for} \hspace*{0.15cm} -M2^M\leq k < M2^M, \nonumber 
\end{cases}
\end{align}
where $M2^{-M+1}<\epsilon$ and $\delta \in(0, \frac1{2^{M+1}})$. Similarly, we define $Q_d$. From these definitions and \eqref{approximate1}, it follows that $(P_d, Q_d)\in \set{P}_{M2^{M+2}}\cap \set{P}[m_P+O(\epsilon), \sigma_P+O(\epsilon); m_Q+O(\epsilon), \sigma_Q+O(\epsilon)]$ and $|\TV(P,Q)-\TV(P_d, Q_d)|<O(\epsilon)$.
By applying \eqref{th_1_1} for finite discrete probability measures $P_d$ and $Q_d$, it follows that $(P,Q)$ also satisfies \eqref{th_1_1} since $\epsilon$ is an arbitrarily small number. This completes the proof for Item (a).

We finally prove Item (c) in Theorem~\ref{th_1}. Since \eqref{th_1_3} is symmetric with respect to $P$ and $Q$, it is sufficient to prove for $\sigma_P>\sigma_Q$. Let $m:=m_P=m_Q$, and let
\begin{align}
P_k(x) := 
\begin{dcases}
\frac{1}{2}-\frac{1}{2k}, & \quad x= m \pm \sigma_Q, \nonumber \\
\frac{1}{2k},       & \quad x =  m \pm \sqrt{(\sigma_P^2-\sigma_Q^2)k+\sigma_Q^2} , \nonumber
\end{dcases}
\end{align}
and 
\begin{align}
Q_k(x) := 
\begin{dcases}
\frac12 , & \quad x = m \pm \sigma_Q,\nonumber \\
0, & \quad x =  m \pm \sqrt{(\sigma_P^2-\sigma_Q^2)k+\sigma_Q^2} \nonumber
\end{dcases}
\end{align}
for sufficiently large $k$.
As $k\rightarrow \infty$, we have $\TV(P_k, Q_k)=\frac 1k\rightarrow 0$. 
\end{proof}

\bibliography{reference_TV} 
\bibliographystyle{myplain} 
\appendices
\section{Proof of Lemma~\ref{lem_constraints}}
\label{pr_appendix}
\begin{proof}
Since the case when $|N_{p,q}|\geq 3$ is trivial, we prove for the case when $|N_{p,q}|\leq 2$.
Suppose that $\sum_{k=1}^6 \alpha_k \nabla_{p,q,x} g_k=0$.
\begin{enumerate}[(A)]
\item 
$|N_{p,q}|=2$: \\
Let $N_{p,q}=\{1,2\}$. If $|N_{q=0}|\geq 1$ and letting $N_{q=0}=\{3, \cdots\}$, for components $\{p_1, p_3, x_3\}$, we have
\begin{align}
(\nabla_{p,x} g_k):=(\nabla_{p,x} g_1, \nabla_{p,x} g_2, \cdots, \nabla_{p,x} g_6)= \nonumber
\begin{pmatrix}
   1 & x_1 & x_1^2 & 0 & 0 & 0 \\
   1 & x_3 & x_3^2 & 0 & 0 & 0 \\
   0 & p_3 & 2p_3x_3 & 0 & 0 & 0 
\end{pmatrix}
.
\end{align}
Therefore, we obtain $\alpha_k=0$ for $1\leq k \leq 3$. Considering components $\{q_1, q_2, x_1\}$, we have $\alpha_k=0$ for $4\leq k \leq 6$. Similarly, one can prove the case when $|N_{p=0}|\geq 1$.
We next consider $|N_{q=0}|=|N_{p=0}|=0$. For components $\{p_1, p_2, q_1, q_2, x_1, x_2\}$, we have 
\begin{align}
A:=(\nabla_{p,q,x} g_k)= \nonumber
\begin{pmatrix}
   1 & x_1 & x_1^2 & 0 & 0 & 0 \\
   1 & x_2 & x_2^2 & 0 & 0 & 0 \\
   0 & 0 & 0 & 1 & x_1 & x_1^2 \\
   0 & 0 & 0 & 1 & x_2 & x_2^2 \\
   0 & p_1 & 2p_1x_1 & 0 & q_1 & 2q_1x_1 \\
   0 & p_2 & 2p_2x_2 & 0 & q_2 & 2q_2x_2 \\
\end{pmatrix}
.
\end{align}
The determinant of the matrix $A$ is $-(p_1q_2-p_2q_1)(x_1-x_2)^4$. Since $P\neq Q$ from the assumption for means, we have $p_1q_2-p_2q_1\neq 0$.
Hence, the result follows.
\item 
$|N_{p,q}|=1$: \\
From the assumption for variances, we have $|N_{q=0}|+|N_{p=q}|\geq 1$ and $|N_{p=0}|+|N_{p=q}|\geq 1$. Since $\sum_i p_i=\sum_i q_i=1$, it must be $|N_{q=0}|\geq 1$ or $|N_{p=0}|\geq 1$. In the similar way to Case (A), we have $\alpha_k=0$ for $1\leq k \leq 6$.
\item 
$|N_{p,q}|=0$: \\
The case when $|N_{p=q}|=0$ is trivial.
From $P\neq Q$ and $\sum_i p_i=\sum_i q_i=1$, we have $|N_{q=0}|\geq 1$ and $|N_{p=0}|\geq 1$.
Letting $N_{q=0}=\{1, \cdots\}$, $N_{p=0}=\{2, \cdots\}$ and $N_{p=q}=\{3, \cdots\}$, we have
\begin{align}
B:=(\nabla_{p,q,x} g_k)= \nonumber
\begin{pmatrix}
   1 & x_1 & x_1^2 & 0 & 0 & 0 \\
   0 & p_1 & 2p_1x_1 & 0 & 0 & 0 \\  
   0 & 0 & 0 & 1 & x_2 & x_2^2 \\
   0 & 0 & 0 & 0 & q_2 & 2q_2x_2 \\
   1 & x_3 & x_3^2 & 1 & x_3 & x_3^2 \\
   0 & p_3 & 2p_3x_3 & 0 & p_3 & 2p_3x_3 \\
\end{pmatrix}
\end{align}
for components $\{p_1, x_1, q_2, x_2, p_3, x_3\}$.
Since $\det(B)=2p_1q_2p_3(x_1-x_2)(x_2-x_3)(x_3-x_1)\neq 0$, we have $\alpha_k=0$ for $1\leq k \leq 6$.
\end{enumerate}
\end{proof}

\section{Derivation of \eqref{pr_1_13}}
\label{deri_appendix}
Let $N_{p,q}=\{1,2\}$ and $N_{q=0}=\{3\}$. From \eqref{pr_1_6} for $s_i=-1$, we have $\phi(x)=\eta(x-x_1)(x-x_2)-1$ and $\psi(x)=\tilde{\eta}(x-x_1)(x-x_2)+1$.
By Lemma~\ref{lem_lagrange}, $\eta$ and $\tilde{\eta}$ are non-zero constants. Substituting these equations into \eqref{pr_1_7}, we have $\frac{p_1}{q_1}=\frac{p_2}{q_2}$. 
Thus, we define probability measures $P$ and $Q$ on $\{x_1,x_2, x_3\}$ as $P=((1-p)(1-q), (1-p)q, p)$ and $Q=(1-q, q,0)$. The moment constraints reduce to 
\begin{align}
\label{ap_3_1}
\begin{cases}
&(1-p) (1-q)x_1 + (1-p) qx_2 + px_3=m_P,   \\ 
&(1-p) (1-q)x_1^2 + (1-p) qx_2^2 + px_3^2=m_P^2+\sigma_P^2,   \\ 
&(1-q) x_1 + qx_2=m_Q,   \\ 
&(1-q)x_1^2 + qx_2^2=m_Q^2+\sigma_Q^2.
\end{cases}
\end{align}
Subtracting the result of multiplying the third equation by $1-p$ from the first equation in \eqref{ap_3_1}, and subtracting the result of multiplying the forth equation by $1-p$ from the second equation in \eqref{ap_3_1}, we have
\begin{align}
\label{ap_3_2}
&px_3=m_P-(1-p)m_Q=a+pm_Q, \\
\label{ap_3_3}
&px_3^2=m_P^2+\sigma_P^2-(1-p)(m_Q^2+\sigma_Q^2)=a(m_P+m_Q)+\sigma_P^2-\sigma_Q^2+p(m_Q^2+\sigma_Q^2).
\end{align}
Subtracting the square of \eqref{ap_3_2} from the result of multiplying \eqref{ap_3_3} by $p$, we have
\begin{align}
\sigma_Q^2p^2+(\sigma_P^2-\sigma_Q^2+a^2)p-a^2=0. \nonumber
\end{align}
Solving this equation for non-negative $p$, we have
\begin{align}
&p=\frac{-(\sigma_P^2-\sigma_Q^2+a^2)+v}{2\sigma_Q^2}=\frac{2a^2}{v+\sigma_P^2-\sigma_Q^2+a^2}. \nonumber
\end{align}
Since $\TV(P,Q)=p$, we obtain \eqref{pr_1_13}.

\end{document}